\documentclass{amsart}
\usepackage{amsfonts}
\usepackage{amsmath,amsthm}
\usepackage{amsfonts,amssymb}
\usepackage{mathrsfs}
\usepackage{enumerate}

\newtheorem{thm}{Theorem}[section]

\newtheorem{lem}[thm]{Lemma}

\theoremstyle{remark}

\numberwithin{equation}{section}

\def\vz{\varepsilon}
\def\oz{\omega}
\def\lz{\lambda}

\def\dz{\delta}
\def\az{\alpha}
\def\gz{\gamma}

\def\tz{\theta}

\def\({\Bigl(}
\def \){ \Bigr)}

\def\ga{\gamma}

 \def\a{{\alpha}}

\def\R{\mathbb{R}}
\def\x{{\bf x}}

\def\va{\varepsilon}

\def\y{{\bf y}}
\def\j{{\bf j}}
\def\gaa{\boldsymbol{\ga}}

\begin{document}

\title[] {Average Case tractability of  multivariate  approximation with Gaussian  kernels}

\author[]{Jia Chen, Heping Wang} \address{School of Mathematical Sciences, Capital Normal
University, Beijing 100048,
 China.}
\email{ jiachencd@163.com;\ \  \ wanghp@cnu.edu.cn.}



\thanks{
 Supported by the
National Natural Science Foundation of China (Project no.
11671271) and
 the  Beijing Natural Science Foundation (1172004)
 }

\begin{abstract} We study the problem of approximating
functions of $d$ variables in the average case setting for the
$L_2$ space $L_{2,d}$ with the standard Gaussian weight equipped
with a zero-mean Gaussian measure. The covariance kernel of this
Gaussian measure takes the form of a Gaussian kernel  with
non-increasing positive shape parameters $\gz_j^2$ for $j = 1, 2,
\dots, d$. The error of approximation is defined in the norm of
$L_{2,d}$.
 We study the average  case error of algorithms
that use at most $n$ arbitrary continuous linear functionals. The
information complexity $n(\vz, d)$ is defined as the minimal
number of linear functionals which are needed to find an
 algorithm whose
average  case error is at most $\vz$.  We study different notions
of tractability or exponentially-convergent tractability
(EC-tractability) which the information complexity $n(\vz, d)$
describe how behaves as a function of $d$ and $\vz^{-1}$ or as one
of $d$ and $(1+\ln\vz^{-1})$.
 We find necessary and sufficient conditions on various notions
of tractability and EC-tractability in terms of shape parameters.
In particular, for any positive $s>0$ and $t\in(0,1)$ we obtain
that the sufficient and necessary  condition on $\gz^2_ j$ for
which
$$\lim_{d+\vz^{-1}\to\infty}\frac{n(\vz,d)}{\vz^{-s}+d^t}=0$$
holds is
$$ \lim_{j\to \infty}j^{1-t}\ga_j^2\,\ln^+ \ga_j^{-2}=0,$$where
$\ln^+ x=\max(1,\ln x)$.

\end{abstract}

\maketitle
\input amssym.def

\section{Introduction and main results}

Recently, there has been an increasing interest in $d$-variate
computational problems with large or even huge $d$.  Examples
include problems in computational finance, statistics and physics.
Such problems are usually solved by algorithms that use finitely
many information operations. The information complexity $n(\vz,
d)$ is defined as the minimal number of information operations
which are needed to find an approximating solution to within an
error threshold $\vz$. A central issue is the study of how the
information complexity depends on $\vz$ and $d$.
 Such problem is called the
tractable problem. There are two kinds of tractability based on
polynomial-convergence and  exponential-convergence. The
(classical) tractability describes how the information complexity
$n(\vz, d)$  behaves as a function of $d$ and $\vz^{-1}$,  while
the exponentially-convergent tractability (EC-tractability)  does
as one of $d$ and $(1+\ln\vz^{-1})$. Nowadays study of
tractability and EC-tractability has become one of the busiest
areas of research in information-based complexity (see \cite{NW1,
NW2, NW3, DLP, IKP, PP, X2} and the references therein).

In this paper, we  consider tractability of a multivariate
approximation problem defined over the space $L_{2,d}$ in the
average case setting, where $$L_{2,d}=\Big\{ f\ \big| \
\|f\|_{L_{2,d}}=\bigg(\int_{\R^d} |f(\x)|^2\prod_{j=1}^d \frac
{\exp {(-x_j^2)}}{\sqrt{\pi}}\,d\x\bigg)^{1/2}<\infty\Big\}$$ is a
separable Hilbert space of real-valued functions on $\R^d$ with
inner product $$\langle f,g\rangle_{L_{2,d}}= \int_{\R^d} f(\x)
g(\x)\prod_{j=1}^d \frac {\exp {(-x_j^2)}}{\sqrt{\pi}}\,d\x.$$ The
space $L_{2,d}$ is equipped with a zero-mean Gaussian measure
$\mu_d$ with Gaussian covariance kernel
\begin{equation}\label{1.0}
K_{d,\gaa}(\x,\y)=\int_{L_{2,d}}f(\x)f(\y)\mu_d(df)=\prod_{j=1}^{d}K_{\gz_j}(x_j,
y_j), \ \x,\y\in \R^d,
\end{equation}
where $$K_{\gz}(x, y)=\exp (-\ga^2(x-y)^2), \ x,y\in \R,$$ and
$\gaa=\{\ga_j^2\}_{j\in \Bbb N}$ is a given sequence of shape
parameters  not depending on $d$ and satisfying
\begin{equation}\label{1.1}
\ga_1^2\geq \ga_2^2\geq\dots>0.
\end{equation}

 We consider multivariate approximation which is defined via the embedding operator
$$ {\rm App}_d: L_{2,d}\to L_{2,d}\ \ {\rm with}\ \  {\rm App}_d\,
f=f.$$ We approximate ${\rm APP}_d \, f$ by algorithms that use
only finitely many continuous linear functionals. A function $f\in
L_{2,d}$ is approximated by an algorithm
\begin{equation}\label{1.2}A_{n,d}(f)=\Phi
_{n,d}(L_1(f),L_2(f),\dots,L_n(f)),\end{equation} where
$L_1,L_2,\dots,L_n$ belong to continuous linear functionals on
$L_{2,d}$,
 and $\Phi _{n,d}:\;\Bbb R^n\to
L_{2,d}$ is an arbitrary measurable mapping. It is well known (see
\cite{NW1}) that we can restrict ourselves to linear algorithms
$A_{n,d}$ of the form \begin{equation}\label{1.3}A_{n,d}f =
\sum_{k=1}^n L_k(f) \psi_k, \end{equation} where $\psi_k\in
L_{2,d},\ k=1,2,\dots,n$. The average case error for  $A_{n,d}$ is
defined by
\begin{equation*}
e(A_{n,d})\;=\;\left ( \int _{L_{2,d}}\left \| {\rm
App}_d\,f-A_{n,d}f \right \|_{L_{2,d}}^{2}\mu _d(df) \right
)^{\frac{1}{2}}.
\end{equation*}

The $n$th minimal average case error, for $n\ge 1$, is defined by
\begin{equation*}
e(n,d)=\inf_{A_{n,d}}e(A_{n,d} ),
\end{equation*}
where the infimum is taken over all algorithms of the form
\eqref{1.2} or \eqref{1.3}. For $n=0$, we use $A_{0,d}=0$. We
remark that  the so-called initial error $e(0,d)$, defined by
\begin{equation*}
e(0,d)=\Big ( \int _{L_{2,d}}\big \| {\rm App}_d\,f \big
\|_{L_{2,d}}^{2}\mu _d(df) \Big )^{\frac{1}{2}},
\end{equation*}is equal to $1$. In other words, the normalized error criterion and the absolute error criterion
coincide.

The information complexity $n(\vz, d)$ is defined by
$$n(\vz,d)=\inf\ \{n\ |\ e(n,d)\le \vz\}.$$

 Let ${\rm App}=\{{\rm
App}_d\}_{d\in \Bbb N}$. First we consider the classical
tractability of App.

 Various notions of (the classical) tractability
have been studied recently for many multivariate problems.  We
briefly recall some of the basic tractability notions (see
\cite{NW1, NW3, S1, SiW}).

 We say App is

$\bullet$   {\it strongly polynomially tractable (SPT)}   iff
there exist non-negative numbers $C$ and $p$ such that for all
$d\in \Bbb N,\ \va \in (0,1)$,
\begin{equation*}
n(\va ,d)\leq C(\va ^{-1})^p;
\end{equation*} The exponent of SPT the exponent is defined to be  the infimum of all $p$ for which
the above inequality holds;

 $\bullet$  {\it polynomially tractable
(PT)} iff there exist non-negative numbers $C, p$ and $q$ such
that for all $d\in \Bbb N, \ \va \in(0,1)$,
\begin{equation*}
n(\va ,d)\leq Cd^q(\va ^{-1})^p;
\end{equation*}

$\bullet$   {\it quasi-polynomially tractable (QPT)} iff there
exist two constants $C,t>0$ such that for all $d\in \Bbb N, \ \va
\in(0,1)$,
\begin{equation*}
n(\va ,d)\leq C\exp(t(1+\ln\va ^{-1})(1+\ln d));
\end{equation*}

$\bullet$ {\it uniformly weakly tractable (UWT)} iff for all $s,
t>0$,
\begin{equation*}
\lim_{\varepsilon ^{-1}+d\rightarrow \infty }\frac{\ln n(\va
,d)}{(\va ^{-1})^{s }+d^{t }}=0;
\end{equation*}

$\bullet$  {\it weakly tractable (WT)} iff
\begin{equation*}
\lim_{\va ^{-1}+d\rightarrow \infty }\frac{\ln n(\va ,d)}{\va
^{-1}+d}=0;
\end{equation*}

$\bullet$ {\it $(s,t)$-weakly tractable ($(s,t)$-WT)} for positive
$s$ and $t$ iff
\begin{equation*}
\lim_{\varepsilon ^{-1}+d\rightarrow \infty }\frac{\ln n(\va
,d)}{(\va ^{-1})^{s }+d^{t }}=0.
\end{equation*}

Clearly, $(1,1)$-WT is the same as WT. If App is not WT, then  App
is called  intractable. We say that ${\rm App}$ suffers from {the
curse of dimensionality} if there exist positive numbers $C, \,\va
_0,\, \a $ such that for all $0<\va\leq \va _{0}$ and infinitely
many $d\in \Bbb N$,
\begin{equation*}
n(\va ,d)\geq C(1+\a )^{d}.
\end{equation*}

  SPT and QPT of
the above approximation problem  App have been studied in
\cite{FHW2} and \cite{K}, respectively. The following conditions
have been obtained therein:

\

 $\bullet$  SPT holds iff there exists  a positive number $\dz>1$ such that
$ \sum\limits_{j=1}^\infty \ga_{j}^{2/\dz}<\infty$ iff $r(\ga)>1$,
where
\begin{equation}\label{1.4} r(\ga)=\sup\,\big\{\dz>0\ | \ \sum_{j=1}^\infty
\ga_{j}^{2/\dz}<\infty\big\}=\sup\,\big\{\beta\ge 0\ |\
\lim_{j\to\infty}j^\beta\gz_j^2=0\big\}.
\end{equation}
 In this case, the exponent  of SPT is
$\frac 2{r(\ga)-1}$.

$\bullet$ QPT holds iff  $$\sup _{d\in \Bbb N}\frac{1}{\ln^+d}\sum
_{j=1}^d \ga _j^2(1+\ln(1+\ga_j^{-2}))<\infty,$$where $\ln^+
x=\max(1,\ln x)$.

\

In this paper we  obtain  complete results about the tractability
of App. Specially, we give the necessary and sufficient condition
for $(s, t)$-WT for $t\in(0,1)$ and $s>0$. Similar  conditions are
first given in our paper.  We use the new method. We remark that
in similar approximation problems with covariance kernels
corresponding to Euler and Wiener integrated processes under the
normalized error criterion, the necessary and sufficient
conditions for $(s, t)$-WT for $t\in (0, 1)$ and $s>0$ do not
completely match (see \cite{S2}).

\begin{thm}\label{thm1} Consider the above  approximation problem
App  with shape parameters $\gaa=\{\gz_j^2\}$
 satisfying \eqref{1.1}.

\vskip 2mm

(i)  PT  holds iff  SPT holds  iff \begin{equation}\label{1.5}
r(\ga)=\underset{j\to\infty}{\underline{\lim}}\frac{\ln
\ga_j^{-2}}{\ln j} >1 .
\end{equation}

(ii)  For  $t>1$ and $s>0$,  $(s,t)$-WT holds for all shape
parameters.

\vskip 2mm

(iii) For $t=1$ and $s>0$, $(s,1)$-WT holds iff WT holds iff
\begin{equation}\label{1.7}
\lim _{j\to \infty}\ga_j^2=0.
\end{equation}

(iv) For  $t\in(0,1)$ and $s>0$,  $(s,t)$-WT holds iff
\begin{equation}\label{1.8}
\lim_{j\to \infty}j^{1-t}\ga_j^2\,\ln^+ \ga_j^{-2}=0.
\end{equation}

(v)  UWT holds iff
\begin{equation}\label{1.9}
\underset{j\to\infty}{\underline{\lim}}\frac{\ln\ga_j^{-2}}{\ln
j}\geq 1.
\end{equation}

(vi)  $App$ suffers from the curse of dimensionality if
$\lim\limits_{j\to \infty}\ga_j^2>0$.
\end{thm}

\vskip 2mm

  It is of interest to compare the tractability results
of Theorem 1.1 with the ones in the worst case setting from
\cite{FHW1}, where the behavior of the information complexity in
the worst case setting  is studied using either {\it the absolute
error criterion (ABS)}  or {\it the normalized error criterion
(NOR)} (see Subsection 2.3 for related notions in the worst case
setting).

\vskip 3mm

For ABS, we have

\vskip 2mm

 $\bullet$  SPT holds for all shape parameters with the
exponent $\min\big(2, \frac{2}{r(\ga)}\big),$ where $r(\gz)$ is
given by \eqref{1.4}.

\vskip 2mm

$\bullet$ Obviously, SPT implies all PT, QPT, WT, $(s,t)$-WT for
any positive $s$ and $t$, as well as UWT, for all shape
parameters.

\vskip 3mm

For NOR, we have

\vskip 2mm

 $\bullet$ SPT holds iff PT  holds iff
$r(\ga)>0$, with the exponent $\frac{2}{r(\ga)}$.

\vskip 2mm

$\bullet$ QPT holds for all shape parameters.

\vskip 2mm

$\bullet$ Obviously, QPT implies $(s,t)$-WT for any positive $s$
and $t$, as well as UWT for all shape parameters.

\vskip 2mm

We stress that there is now a difference between ABS and NOR in
the worst case setting. For all shape parameters, we always have
SPT for ABS and  QPT for NOR in the worst case setting, whereas in
the average case setting we only have $(s,t)$-WT for $s>0$ and
$t>1$. Also the sufficient and necessary condition for SPT (or PT)
in the average case setting is stronger than the one for NOR in
the worst case setting.

\vskip 3mm

Next we consider  exponential convergence tractability of the
approximation problem App. Because the covariance kernel function
of the Gaussian measure $\mu_d$ is an  analytic function, the
$n$th minimal error $e(n,d)$ can be expected to decay faster than
any polynomial. Indeed, we expect exponential convergence.
 If there exists a number $q\in(0,1)$ such that for
all $d=1,2,...,$ there are positive numbers $C_{1,d}$, $C_{2,d}$
and $p_d$ for which
\begin{equation}\label{1.10}
e(n,d)\le C_{1,d}q^{(n/C_{2,d})^{p_d}},\quad \text{for all }\,
n\in \Bbb N,
\end{equation} then we say that App is \emph{exponential convergence
(EXP)}.
  The supremum of positive $p_d$ in \eqref{1.10} is called the
exponent of EXP. If $p_d$ can be chosen as positive and
independent of $d$ we have \emph{uniform exponential convergence
(UEXP)}.

If App is EXP, then we can discuss the {\it tractability with
exponential convergence (EC-tractabilty)}. Recently, there are
many papers where EC-tractability  is considered (see \cite{DLP,
IKP, PP, X2}).

 In the definitions of SPT, PT, QPT, UWT, WT, and
$(s,t)$-WT, if we replace $\frac1{\vz}$ by $(1+\ln \frac 1{\vz})$,
we get the definitions of \emph{exponential convergence-strong
polynomial tractability (EC-SPT)}, \emph{exponential
convergence-polynomial tractability (EC-PT)}, \emph{exponential
convergence-quasi-polynomial tractability (EC-QPT)},
\emph{exponential convergence-uniform weak tractability
 (EC-UWT)}, \emph{exponential convergence-weak tractability
 (EC-WT)}, and \emph{exponential convergence-$(s,t)$-weak tractability
 (EC-$(s,t)$-WT)}, respectively.

\vskip 2mm

In \cite{SW}, Sloan and Wo\'zniakowski obtained the following
complete results about the EC-tractability
 in the worst case setting using  ABS and NOR.

\vskip 2mm

For ABS or NOR, we have

\vskip 2mm

$\bullet$ EXP  holds with the exponent $p^*_d=1/d$  and  UEXP does
not hold for all shape parameters $\gaa$ satisfying \eqref{1.1}.

\vskip 2mm

$\bullet$  EC-SPT and EC-PT and EC-QPT do not hold for all shape
parameters.

\vskip 2mm

$\bullet$ If $\max(s,t)>1$ then EC-$(s,t)$-WT holds for all shape
parameters.

\vskip 2mm

$\bullet$ EC-WT holds iff $\lim\limits_{j\to \infty}\ga^2_j=0$.

\vskip 2mm

$\bullet$ EC-$(1,t)$-WT with $t<1$ holds iff
$\lim\limits_{j\to\infty}\frac{\ln j}{\ln \ga_j^{-2}}=0$.

\vskip 2mm

$\bullet$ EC-$(s,t)$-WT with $s<1$ and $t\leq 1$  holds iff
$\lim\limits_{j\to\infty}\frac{j^{(1-s)/s}}{\ln \ga_j^{-2}}=0$.

\vskip 2mm

$\bullet$  EC-UWT holds iff $\lim\limits_{j\to\infty}\frac{\ln
(\ln\ga_j^{-2})}{\ln j}=\infty$.

\vskip 2mm

EC-tractability in the worst and average case settings has the
intimate connection. Specially, according to \cite[Theorems 3.2
and 4.2]{X2} and \cite[Theorem 3.2]{LXD}, we have the same results
in the worst and average case settings using ABS concerning EC-WT,
EC-UWT, and EC-$(s,t)$-WT for $0<s\le 1$ and $t>0$.

 Based on the results of \cite{SW}, we get the EC-tractability of
App in the average case setting.

\begin{thm}\label{thm2} Consider the above  approximation problem
App  with shape parameters $\gaa=\{\gz_j^2\}$
 satisfying \eqref{1.1}.

\vskip 2mm

(i) EXP  holds with the exponent $p^*_d=1/d$  and  UEXP does not
hold for all shape parameters $\gaa$ satisfying \eqref{1.1}.

\vskip 2mm

(ii)  EC-SPT and EC-PT and EC-QPT do not hold for all shape
parameters.

\vskip 2mm

(iii)  If $s>0$ and $t>1$ then EC-$(s,t)$-WT holds for all shape
parameters.

\vskip 2mm

(iv) EC-$(s,1)$-WT with $s\ge 1$ holds iff EC-WT holds iff
$\lim\limits_{j\to \infty}\ga^2_j=0.$

\vskip 2mm

(v) EC-$(s,t)$-WT with $s<1$ and $t\leq 1$  holds iff
$$\lim\limits_{j\to\infty}\frac{j^{(1-s)/s}}{\ln \ga_j^{-2}}=0.$$

(vi) EC-$(1,t)$-WT with $t<1$ holds iff
$$\lim\limits_{j\to\infty}\frac{\ln j}{\ln \ga_j^{-2}}=0.$$

 (vii) EC-$(s,t)$-WT with $s>1$ and $t< 1$ holds iff
\begin{equation}\label{1.12}\lim_{j\to
\infty}j^{1-t}\ga_j^2\,\ln^+
\ga_j^{-2}=0.\end{equation}

(viii)  EC-UWT holds iff
$$\lim\limits_{j\to\infty}\frac{\ln(\ln\ga_j^{-2})}{\ln
j}=\infty.$$
\end{thm}

Let us compare the results about EC-tractability in the worst and
average case settings. There are the same conclusion for EC-SPT,
EC-PT, EC-QPT, EC-UWT, WT, and EC-$(s,t)$-WT with $s\le1, \ t>0$
or $s>0, t>1$ in the worst and average case settings.  We always
have EC-$(s,t)$-WT with $s>1,\ 0<t\le1$ for all shape parameters
in the worst  case setting, whereas in the  average case setting,
EC-$(s,t)$-WT with $s>1,\ 0<t\le1$ holds iff $(s,t)$-WT with
$s>1,\ 0<t\le1$ holds iff
$$\lim_{j\to \infty}j^{1-t}\ga_j^2\,\ln^+ \ga_j^{-2}=0.$$

We also compare the results about tractability and EC-tractability
in the average case setting.  We never have  EC-SPT, EC-PT,
EC-QPT, whereas SPT, PT, QPT hold for shape parameters decaying
fast enough. For all shape parameters, we always have
EC-$(s,t)$-WT and $(s,t)$-WT for $s>0$ and $t>1$. There are the
same sufficient and necessary conditions for which EC-$(s,t)$-WT
and  $(s,t)$-WT hold with $s>1, \ 0\le t\le1$ or $s=t=1$. In the
other cases, we need to assume more conditions about shape
parameters to get EC-UWT or EC-$(s,t)$-WT than ones to get UWT or
$(s,t)$-WT with $0<s<1,\ 0<t\le 1$ or $s=1, 0<t<1$.

The paper is organized as follows. In Subsection 2.1 we give
concept
 of non-homogeneous tensor product problems in the
average case setting.
   Subsections 2.2 and 2.3 are
devoted to introducing the average and worst case approximation
problems  with Gaussian  kernels.
  In  Section 3, we give the proofs of Theorems 1.1 and 1.2.

\section{Preliminaries}

\subsection{Average case non-homogeneous tensor product
problems}

\

We  recall the concept of non-homogeneous tensor product problems,
see \cite{LPW}. Let $F_d, H_d$  are given by tensor products. That
is,
\begin{equation*}
F_d=F^{(1)}_1\otimes F^{(1)}_2\otimes \dots \otimes F^{(1)}_d
\quad \text {and} \quad H_d=H^{(1)}_1\otimes  H^{(1)}_2\otimes
\dots \otimes H^{(1)}_d,
\end{equation*}
where Banach spaces $F^{(1)}_k$ are of univariate real functions
equipped with a zero-mean Gaussian measure $\mu^{(1)}_k$, and
$H^{(1)}_k$ are Hilbert spaces, $k=1,2,\dots,d$. We set
$$S_d=S^{(1)}_1\otimes S^{(1)}_2\otimes \dots \otimes S^{(1)}_d,\ \ \mu_d=\mu_1^{(1)}\otimes \mu_2^{(1)}\otimes \dots \otimes \mu_d^{(1)},$$
where
$$S^{(1)}_k=F^{(1)}_k\to H^{(1)}_k,\quad k=1,2,\dots,d$$ are continuous linear
operators. Then $\mu_d$ is a zero-mean Gaussian measure on $F_d$
with covariance operator $C_{\mu_d}: F_d^*\to F_d$.

Let $\nu_d =\mu_d (S_d)^{-1}$ be the induced measure. Then $\nu_d$
is a zero-mean Gaussian measure on $H_d$ with covariance operator
$C_{\nu_d}: H_d\to H_d$ given by $$ C_{\nu_d}=S_d
\,C_{\mu_d}\,S_d^*,$$where $S_d^*:H_d\to F_d^*$ is the operator
dual to $S_d$.
 Let
$\nu^{(1)}_k=\mu^{(1)}_k(S^{(1)}_k)^{-1}$ be the induced zero-mean
Gaussian measure on $H^{(1)}_k$, and let
$C_{\nu^{(1)}_k}:H^{(1)}_k\to H^{(1)}_k$ be the covariance
operator of the measure $\nu^{(1)}_k$. Then
$$ \nu_d=\nu_1^{(1)}\otimes \nu_2^{(1)}\otimes \dots \otimes \nu_d^{(1)},\ \ {\rm and}\ \ C_{\nu _d}=C_{\nu^{(1)}_1}\otimes C_{\nu^{(1)}_2}\otimes \dots C_{\nu^{(1)}_d}.$$
 The eigenpairs of
$C_{\nu^{(1)}_k}$ are denoted by $\big\{(\lz(k,j),
\eta(k,j))\big\}_{j\in \Bbb N}$, and satisfy
$$C_{\nu^{(1)}_k}(\eta(k,j))=\lz(k,j)\eta(k,j), \ {\rm with}\ \lz(k,1)\geq \lz(k,2)\geq \dots \geq 0.$$
Then  $${\rm
trace}(C_{\nu^{(1)}_k})=\int_{H^{(1)}_k}\|f\|^2_{H^{(1)}_k}\nu^{(1)}_k(df)
=\sum_{j=1}^\infty \lz(k,j) <\infty.$$ The eigenpairs of $C_{\nu
_d}$ are given by
$$\big\{(\lz _{d,\j},\eta _{d,\j})\big\}_{\j=(j_1,j_2,\dots,j_d)\in \Bbb N^d},$$
where
$$\lz_{d,\j}=\prod _{k=1}^d\lz (k,j_k)\quad \text{and}\quad \eta_{d,\j}=\prod_{k=1}^d\eta(k,j_k).$$

Let the sequence $\{\lz _{d,j}\}_{j\in \Bbb N}$ be the
non-increasing rearrangement of $\{\lz_{d,\j}\}_{\j\in \Bbb N^d}$.
Then we obtain
\begin{equation*}
\sum_{j\in \Bbb N}\lz^{\tau}_{d,j}=\prod_{k=1}^d\sum_{j=1}^\infty
\lz(k,j)^\tau,\quad \text{for any}\quad \tau>0.
\end{equation*}

We approximate $S_d \, f$ by algorithms $A_{n,d}$ of the form
\eqref{1.2} that use only finitely many continuous linear
functionals on $ F_d$. Then the  $n$th minimal average  case error
is given by
$$e(n,S_d):=\inf_{A_{n,d}} \;\Big ( \int _{F_d}\big \| S_d\,f
-A_{n,d}f\big \|_{H_d}^{2}\mu _d(df) \Big )^{\frac{1}{2}}= \Big
(\sum _{j=n+1}^{\infty}\lambda _{d,j}  \Big )^{\frac{1}{2}},$$ and
is achieved by the $n$th optimal algorithm
\begin{equation*}
A_{n,d}^{*}(f)=\sum_{j=1}^n  \big\langle f,\eta _{d,j} \big
\rangle _{H_d}\eta _{d,j}.\end{equation*} The initial error for
$S_d$ is
$$  e(0,S_d)=\Big ( \int _{F_d}\big \| S_d\,f \big
\|_{H_d}^{2}\mu _d(df) \Big )^{\frac{1}{2}}=\Big (\sum
_{j=1}^{\infty}\lambda _{d,j}  \Big )^{\frac{1}{2}}.$$
 The information
complexity for $S_d$ can be studied using either the absolute
error criterion (ABS), or the normalized error criterion (NOR).
Then we define the information complexity $n^X(\va ,S_d)$ for
$X\in \{ {\rm ABS,\, NOR}\}$ as
\begin{equation*}
n^{ X}(\va ,S_d)=\min\{n:\,e(n,S_d)\leq \va CRI_d\},
\end{equation*}where
\begin{equation*}
CRI_d=\left\{\begin{matrix}
 &\ 1, \; \qquad\text{ for X=ABS,} \\
 &e(0,S_d),\ \text{ for X=NOR.}
\end{matrix}\right.
\end{equation*}

In order to prove Theorem 1.1, we need the following lemmas.

\begin{lem}\label{lem2.1} (See \cite[Theorem 6]{LPW}.)
Let $S=\{S_d\}$ be a non-homogeneous tensor product problem. Then
for NOR, S is PT if and only if there exists $\tau \in (0,1)$ such
that
\begin{equation*}
Q_\tau:=\sup_{d\in \Bbb N}\frac{1}{\ln ^+d}\sum
_{k=1}^d\ln\Big(1+\sum
_{j=2}^\infty\Big(\frac{\lz(k,j)}{\lz(k,1)}\Big)^\tau\Big)<\infty.
\end{equation*}
\end{lem}

\begin{lem}\label{lem2.2} (See \cite[Theorem 8]{LPW} or
\cite[Lemma 2.1]{S2}.) Let $S=\{S_d\}$ be a non-homogeneous tensor
product problem. If for $t>0$ there exists a number $\tau
\in(0,1)$ such that
\begin{equation*}
\lim_{d\to\infty}\frac{1}{d^t}\sum_{k=1}^d\sum_{j=2}^\infty\left(\frac{\lz(k,j)}{\lz(k,1)}\right)^\tau=0,
\end{equation*}
then $S$ is $(s,t)$-WT for this $t$ and every $s>0$.
\end{lem}

\subsection{Average case approximation problems  with   Gaussian  kernels }

\

Let $C_{\mu_d}$ be the covariance operator of $\mu _d$ with
Gaussian  covariance kernel $K_{d,\gaa}$ given by \eqref{1.0},
where $\mu_d$ is a zero-mean Gaussian measure on $L_{2,d}$. Then
for $f\in L_{2,d}$,
$$C_{\mu_d}(f)(\x)=\int_{\R^d}
K_{d,\gaa}(\x,\y)f(\y)\prod_{k=1}^d\frac{\exp(-y_j^2)}{\sqrt\pi}\,d\y,\
\ \x\in\R^d.$$

First, we consider the case $d=1$.  Let $C_{\mu_1}$ be the
covariance operator of $\mu _1$ with covariance kernel
$K_{1,\ga}$, and let  $\{(\lz_{\ga, j}, \eta
_{\ga,j})\}_{j=1}^\infty$ be the sequence of eigenpairs of the
covariance operator $C_{\mu_1}$, i.e.,
$$C_{\mu_1}\eta_{\gz,j}=\lz_{\ga, j} \eta
_{\ga,j}, \ \ j=1,2,\dots, $$ and  $$\lz_{\gz,1}\ge \lz_{\gz,2}\ge
\dots\ge 0. $$

Specifically we have, see e.g.,
 \cite[Section 4.3.1]{RW} and \cite{FHW2, SW},
\begin{equation}\label{2.3}
\lz_{\ga, j}=(1-\oz_\ga)\oz_\ga^{j-1},\quad \text{with}\quad
\oz_\ga=\frac{2\ga^2}{1+2\ga^2+\sqrt{1+4\ga^2}},
\end{equation}
and
\begin{equation*}
\eta_{\ga,
j}(x)=\sqrt{\frac{(1+4\ga^2)^{1/4}}{2^{j-1}(j-1)!}}\exp\left(-\frac{2\ga^2x^2}{1+\sqrt{1+4\ga^2}}\right)H_{j-1}((1+4\ga^2)^{1/4}x),
\end{equation*}
where $H_{j-1}$ is the standard Hermite polynomial of degree
$j-1$, i.e.,
$$H_{j-1}(x)=(-1)^{j-1}e^{x^2}\frac{d^{j-1}}{dx^{j-1}}e^{-x^2}\quad \text{for all}\quad x\in \R.$$
Clearly, $0<\oz_\ga<1$,  $$ 1-\oz_\ga=\lz_{\ga, 1}>\lz_{\ga,
2}>\dots>0,\  \ {\rm and}\ \ \sum_{j=1}^\infty \lz_{\ga, j}=1.$$
It follows from \eqref{2.3} that $\oz_\ga$ is an increasing
function of $\ga$ and $\oz_\ga$ tends to $0$ iff $\ga$ tends to
$0$. We also have
\begin{equation}\label{2.4}\lim\limits_{\gz\to0} \frac {\oz_\ga}{\gz^2}=1, \ \ {\rm and}\ \ \lim_{\gz\to0} \frac
{\ln \oz_\ga^{-1}}{\ln \gz^{-2}}=1.\end{equation}

 Due to the tensor product structure of  the covariance
operator $C_{\mu_d}$, the eigenvalues and the corresponding
eigenfunctions  of  $C_{\mu _d}$  have the form
\begin{equation}\label{2.5}
\lz_{\gaa,\j}=\prod_{k=1}^{d}\lz_{\ga_k,j_k}=\prod_{k=1}^{d}[(1-\oz_{\ga_k})\oz_{\ga_k}^{j_k-1})]\quad
\text{and}\quad
\eta_{\gaa,\j}=\prod_{k=1}^{d}\eta_{\ga_k,j_k}(x_k),
\end{equation}
for $\j=(j_1,j_2,\dots,j_d)\in \Bbb N^d$. Let  $\{\lz
_{d,j}\}_{j\in \Bbb N}$ be the non-increasing rearrangement of
$\{\lz_{\gaa,\j}\}_{\j\in \Bbb N^d}$. Then we have
\begin{equation}\label{2.6}
\lz _{d,1}=\prod_{k=1}^{d}(1-\oz_{\ga_k}),\end{equation} and
$$e(0,d)=\sum_{k=1}^\infty \lz _{d,k}= \sum_{\j\in \Bbb N^d}
\lz_{\gaa,\j}=\prod_{k=1}^d\sum_{j=1}^\infty \lz_{\ga_k, j}=1,
$$where $e(0,d)$ is the initial error.
This  means that the normalized error criterion and the absolute
error criterion are the same.

 The  $n$th
minimal average  case error is
$$e(n,d)=\Big (\sum _{j=n+1}^{\infty}\lambda _{d,j}  \Big )^{\frac{1}{2}}.$$

The information complexity $n(\vz, d)$ of the approximation
problem App is defined by
$$n(\vz,d)=\inf\ \{n\ |\ e(n,d)\le \vz\}.$$

We emphasize that App is a non-homogeneous tensor product problem
with
\begin{equation}\label{2.7}
\lz (k,j)=(1-\oz_{\ga_k})\oz_{\ga_k}^{j-1}\quad\text{and}\quad
\lz(k,1)=1-\oz_{\ga_k}, \ \ k=1,\dots,d.
\end{equation}

\subsection{Worst case approximation problems with Gaussian kernels}

\

Let $H(K_{d,\gaa})$ be the reproducing kernel Hilbert space with
the kernel $K_{d,\gaa}$ given by \eqref{1.0}. The function space
$H(K_{d,\gaa})$ has been used widely in  numerical computation,
statistical learning, and  engineering (see e.g., \cite{FHW1,
FHW2, RW, FSK}). We consider multivariate approximation problem
$I=\{I_d\}_{d\in \Bbb N}$ which is defined via the embedding
operator
$$ {I}_d: H(K_{d,\gaa})\to L_{2,d}\ \ {\rm with}\ \  {I}_d\,
f=f.$$ We approximate $I_d$ by algorithms that use only finitely
many continuous linear functionals on $H(K_{d,\gaa})$. The worst
case error of approximation by an algorithm $A_{n,d}$ of the form
\eqref{1.2} or \eqref{1.3} is defined as
 $$e^{\rm wor}(A_{n,d})=\sup_{\|f\|_{H(K_{d,\gaa})}\le1}
\|I_d(f)-A_{n,d}(f)\|_{L_{2,d}}.$$ The $n$th minimal worst case
error, for $n\ge 1$, is defined by
$$e^{\rm wor}(n,d)=\inf_{A_{n,d}}e^{\rm wor}(A_{n,d}),$$
where the infimum is taken over all algorithms of the form
\eqref{1.2} or \eqref{1.3} using $n$ information operators $L_1,
L_2, \dots, L_n \in H(K_{d,\gaa})^*$. The error of $A_{0,d}$ is
called the initial error and is given by
$$e^{\rm wor}(0,d )=\sup_{\|f\|_{H(K_{d,\gaa})}\le1}\|I_d f\|_{L_{2,d}}=\|I_d\|.$$

Let $\lz_{d,j},\ {j\in \Bbb N}$ be the eigenvalues of the
covariance operator $C_{\mu_d}$ of the Gaussian measure $\mu_d$
satisfying
$$\lz_{d,1}\ge \lz_{d,2}\ge\dots\ge\lz_{d,k}\ge \dots>0.$$
Then   the $n$th minimal worst case error $e^{\rm wor}(n,d)$  and
the $n$th minimal average  case error $e(n,d)$ are of forms (see
\cite{NW1})
$$e^{\rm wor}(n,d)=\lz_{d,n+1}^{1/2},\ \ {\rm and}\ \  e(n,d)=\Big(\sum_{k=n+1}^\infty
\lz_{d,k}\Big)^{1/2}\ge e^{\rm wor}(n,d).$$

The worst case information complexity  can be studied using either
ABS or NOR. Then we define the worst case information complexity
$n^{\rm wor, X}(\va ,d)$ for $X\in \{{\rm ABS,\, NOR}\}$ as
\begin{equation*}
n^{\rm wor, X}(\va ,d)=\min\{n:\,e^{\rm wor}(n,d)\leq \va CRI_d\},
\end{equation*}where
\begin{equation*}
CRI_d=\left\{\begin{split}
 &\ \ 1, \; \quad\qquad\text{ for X=ABS,} \\
 &e^{\rm wor}(0,d), \text{ for X=NOR}
\end{split}\right. \ \ =\ \ \left\{\begin{split}
 &\ 1, \; \quad\text{ for X=ABS,} \\
 &\lz_{d,1}^{1/2},\ \text{ for X=NOR.}
\end{split}\right.
\end{equation*}

 Obviously, we have \begin{equation}\label{2.8} n^{\rm wor, ABS}(\va ,d)\le
 n(\vz,d).\end{equation}

\section{Proofs of Theorems \ref{thm1} and \ref{thm2}}

First we give two auxiliary lemmas.

\begin{lem}\label{lem3.1}
Let $\gaa=\{\ga_j^2\}_{j\in \Bbb N}$ satisfy \eqref{1.1} and
$r(\gaa)>0$. Then
\begin{equation}\label{3.1}
r(\gaa)=\underset{j\to\infty}{\underline{\lim}}\frac{\ln
\ga_j^{-2}}{\ln j},
\end{equation}
where $r(\ga)$ is given by \eqref{1.4}.
\end{lem}
\begin{proof}
 Since $r(\gaa)>0$,  there exists a positive $\dz$ such that
\begin{equation}\label{3.2}
M_\dz=\sum_{j=1}^\infty\ga_j^{\frac{2}{\dz}}<\infty.
\end{equation}
It follows that
\begin{equation*}
j\ga_j^{\frac{2}{\dz}}\leq\sum_{k=1}^j\ga_k^{\frac{2}{\dz}}\leq
M_\dz.
\end{equation*}
We have
\begin{equation*}
\ln j-\frac{1}{\dz}\ln\ga_j^{-2}\leq\ln M_\dz,
\end{equation*}
which yields
\begin{equation*}
\frac{\ln\ga_j^{-2}}{\ln j}\geq \dz(\frac{\ln j-\ln M_\dz}{\ln
j}).
\end{equation*}
Letting $j\to \infty$ in the above inequality, we conclude that
\begin{equation*}
\underset{j\to\infty}{\underline{\lim}}\frac{\ln \ga_j^{-2}}{\ln
j}\geq\dz.
\end{equation*}
Taking the supremum over all $\dz$ for which \eqref{3.2} holds, we
get
\begin{equation}\label{3.3}
r(\gaa)\leq  \underset{j\to\infty}{\underline{\lim}}\frac{\ln
\ga_j^{-2}}{\ln j}.
\end{equation}

On the other hand, by \eqref{3.3} we know that
$$\az:=\underset{j\to\infty}{\underline{\lim}}\frac{\ln
\ga_j^{-2}}{\ln j}>0.$$ Then for an arbitrary $\vz\in(0,\az/2)$,
there exists an integer $N>0$ such that for all $j\ge N$ we have
$$\frac{\ln \ga_j^{-2}}{\ln j}\geq \az-\vz.$$
This implies that
$$ \ga_j^{2}\leq j^{-(\az-\vz)}.$$
Choosing $\dz=\az-2\vz>0$ and noting that
$\frac{\az-\vz}{\az-2\vz}>1$, we obtain that
\begin{equation*}
\sum_{j=N}^\infty\ga_j^{\frac{2}{\dz}}\leq\sum_{j=N}^\infty
j^{-\frac{\az(\ga)-\vz}{\az(\ga)-2\vz}}<\infty,
\end{equation*}
and so $\sum\limits_{j=1}^\infty\ga_j^{\frac{2}{\dz}}<\infty.$ It
follows from the definition of $r(\gaa)$ that
\begin{equation*}\dz=\az-2\vz\leq r(\gaa).\end{equation*}

Letting $\vz\to 0$ in the above inequality, we conclude that
\begin{equation*}
r(\gaa)\geq \az=\underset{j\to\infty}{\underline{\lim}}\frac{\ln
\ga_j^{-2}}{\ln j},
\end{equation*}which combining with \eqref{3.3}, gives
\eqref{3.1}.
 Lemma \ref{lem3.1} is proved.
\end{proof}

\begin{lem}\label{lem3.2}
Let $\gaa=\{\ga_j^2\}_{j\in \Bbb N}$ satisfy  \eqref{1.1}. Then
\begin{equation}\label{3.4}
\underset{j\to\infty}{\underline{\lim}}\frac{\ln \ga_j^{-2}}{\ln
j}\ge 1,
\end{equation} iff for any $t\in (0,1)$,  \begin{equation}\label{3.5}\lim_{j\to
\infty}j^{1-t}\ga_j^2\,\ln^+ \ga_j^{-2}=0,\end{equation}iff for
any $t\in (0,1)$,
\begin{equation}\label{3.6} \lim_{j\to
\infty}j^{1-t}\ga_j^2=0.\end{equation}
\end{lem}
\begin{proof} Suppose that \eqref{3.4} holds. Then for any $t\in
(0,1)$, we have for sufficiently large $j$,
$$\frac{\ln
\ga_j^{-2}}{\ln j}>1-t/2,$$ which yields that
$$\ga_j^2<j^{t/2-1}\ \ \ {\rm and}\ \ \ \ \ga_j^2\,\ln^+
\ga_j^{-2}< (1-t/2)j^{t/2-1}\ln j,$$ where in the last inequality,
we used the monotonicity of the function $h(x)=x\ln 1/x,\
x\in(0,1/e)$. Then \eqref{3.6} and \eqref{3.5} follow from the
above inequalities immediately.

On the other hand, we suppose that for any $t\in (0,1)$,
\eqref{3.5} or \eqref{3.6}  holds. Noting that we can deduce
\eqref{3.6} from \eqref{3.5}. So \eqref{3.6} holds. For any $t\in
(0,1)$, we have for sufficiently large $j$, $$j^{1-t}\ga_j^2\le
1,$$which implies that $$\frac{\ln \ga_j^{-2}}{\ln j}\ge 1-t.$$
Letting $j\to \infty$ and then $t\to 0+$, we get \eqref{3.4}.
Lemma 3.2 is proved.
\end{proof}

\noindent{\it \textbf{Proof of Theorem \ref{thm1}}.}

(i) It was proved in \cite{FHW2} that  SPT holds for App iff
$r(\gaa)>1$. Clearly, if SPT holds, then PT holds. So in order to
prove (i), by Lemma 3.1 it suffices to show \eqref{1.5} whenever
PT holds.

Assume that PT holds. According to Lemma 2.1 and \eqref{2.7},
there exists a $\tau \in (0,1)$ such that
\begin{equation}\label{3.7} Q_\tau:=\sup_{d\in \Bbb N}\frac{1}{\ln ^+d}\sum
_{k=1}^d\ln\Big(1+\sum
_{j=2}^\infty\oz_{\gz_k}^{(j-1)\tau}\Big)=\sup_{d\in \Bbb
N}\frac{1}{\ln ^+d}\sum
_{k=1}^d\ln\Big(\frac1{1-\oz_{\ga_k}^\tau}\Big)<\infty.
\end{equation}
 Noting that  the function $\varphi(x)=\ln(1-x)+x$ is decreasing in $(0,1)$ due to the fact that $\varphi'(x)=\frac {-x}{1-x}<0$ and $\varphi(0)=0$, we get
 $$\ln \frac1{1-x}>x.$$  This implies that
\begin{equation}\label{3.8}
\ln\left(\frac{1}{1-\oz_{\ga_k}^\tau}\right)>\oz_{\ga_k}^\tau,\quad
\tau \in (0,1).
\end{equation}
It follows from \eqref{3.7} and  \eqref{3.8} that
 \begin{align} \label{3.9}d\oz_{\ga_d}^\tau\le \sum
 _{k=1}^d \oz_{\ga_k}^\tau\le \sum
 _{k=1}^d\ln\Big(\frac{1}{1-\oz_{\ga_k}^\tau}\Big)\le Q_\tau \ln
 ^+d.
  \end{align}
By \eqref{3.9} we obtain further
\begin{equation*}
\frac{\ln\oz_{\ga_d}^{-1}}{\ln d}\geq \frac{\ln d-\ln (\ln^+
d)-\ln Q_\tau}{\tau \ln d}.
\end{equation*}
Letting $d\to \infty$, we get
\begin{equation}\label{3.10}
\underset{d\to\infty}{\underline{\lim}}\frac{\ln
\oz_{\ga_d}^{-1}}{\ln d}\geq \frac{1}{\tau}>1.
\end{equation}
By \eqref{3.9} we have $ \lim\limits_{d\to \infty}\oz_{\ga_d}=0. $
It follows from \eqref{2.4} and \eqref{3.10} that
\begin{equation*}
\underset{d\to\infty}{\underline{\lim}}\frac{\ln \ga_d^{-2}}{\ln
d}=\underset{d\to\infty}{\underline{\lim}}\frac{\ln
\oz_{\ga_d}^{-1}}{\ln d}>1,
\end{equation*}
which completes the proof of (i).

\vskip 2mm

(ii) Let $t>1$ and $s>0$. By \eqref{2.7} and the Stolz theorem we
have for any $\tau\in(0,1)$, \begin{align}\label{3.11}0\le
\lim\limits_{d\to
\infty}d^{-t}\sum_{k=1}^d\sum_{j=2}^\infty\Big(\frac{\lz(k,j)}{\lz(k,1)}\Big)^\tau
&=\lim_{d\to
\infty}\frac{\sum_{k=1}^d\frac{\oz_{\ga_k}^{\tau}}{1-\oz_{\ga_k}^{\tau}}}{d^t}
 =\lim_{d\to
\infty}\frac{\frac{\oz_{\ga_d}^{\tau}}{1-\oz_{\ga_d}^{\tau}}}{d^{t-1}}\\&
\notag \le \lim_{d\to
\infty}d^{1-t}{\frac{\oz_{\ga_1}^{\tau}}{1-\oz_{\ga_1}^{\tau}}}=0,
\end{align}where in the last inequality we used the monotonicity of the function $h(x)=\frac x{1-x},\ x\in(0,1)$.   By Lemma 2.2, we get that $(s,t)$-WT holds for
 $t>1$ and $s>0$. (ii) is proved.

\vskip 2mm

(iii) Let $t=1$ and $s>0$. If  $\lim\limits_{j\to
\infty}\ga_j^2=0$, then  by \eqref{2.4} we have for any $\tau\in
(0,1)$,
$$\lim_{d\to
\infty}{\frac{\oz_{\ga_d}^{\tau}}{1-\oz_{\ga_d}^{\tau}}}=0. $$
Similar to \eqref{3.11} , we get $$\lim\limits_{d\to
\infty}d^{-1}\sum_{k=1}^d\sum_{j=2}^\infty\Big(\frac{\lz(k,j)}{\lz(k,1)}\Big)^\tau
=\lim_{d\to
\infty}\frac{\sum_{k=1}^d\frac{\oz_{\ga_k}^{\tau}}{1-\oz_{\ga_k}^{\tau}}}{d}
 =\lim_{d\to
\infty}{\frac{\oz_{\ga_d}^{\tau}}{1-\oz_{\ga_d}^{\tau}}}=0.$$ By
Lemma 2.2, we know that $(s,1)$-WT holds for
   $s>0$.

   On the other hand, we suppose that $(s,1)$-WT holds for some
   $s>0$. We want to show that $\lim\limits_{j\to \infty}\ga_j^2=0$.
 It follows from the
definition of $n(\va, d)$ that
$$1-\sum_{k=1}^{n(\va, d)}\lz_{d,k}=\sum_{k=n(\va, d)+1}^\infty\lz_{d,k}\le\va^2.$$
We have
\begin{equation}\label{3.12-0}
1-\va^2\le\sum_{k=1}^{n(\va, d)}\lz_{d,k}\le n(\va, d)\lz_{d,1}.
\end{equation}
This implies that
\begin{equation}\label{3.12}
\begin{split}
\ln n(\va, d)&\ge\ln(1-\va^2)+\ln \lz_{d,1}^{-1}\ge\ln(1-\va^2)+\sum_{k=1}^d\ln\big(\frac{1}{1-\oz_k}\big)\\
&\ge\ln(1-\va^2)+d\ln\big(\frac{1}{1-\oz_d}\big)\ge\ln(1-\va^2)+d\oz_d,
\end{split}
\end{equation}
where in the last step, we used the inequality
$\ln\big(\frac{1}{1-x}\big)\ge x$ for $x\in [0, 1)$. Since
$(s,1)$-WT holds for some $s>0$, by \eqref{3.12} we get
$$0=\lim_{d\to\infty}\frac{\ln(n(\frac{1}{2}, d))}{\big(\frac{1}{2}\big)^s+d}\ge \lim\limits_{d\to\infty}\frac{\ln\frac{3}{4}+d\oz_d}{d}=\lim_{d\to\infty}\oz_d\ge 0,$$
which implies $\lim\limits_{j\to \infty}\ga_j^2=0$. This completes
the proof of (iii).

\vskip 2mm

(iv) Suppose that $(s, t)$-WT holds for $s>0$ and $t\in (0, 1)$.
We want to show that \eqref{1.8} holds. First we show
$\lim\limits_{d\to \infty}d^{1-t}\oz_d=0$. Since $(s,t)$-WT holds
for $s>0$ and $t\in(0,1)$, by \eqref{3.12} we get
$$0=\lim_{d\to\infty}\frac{\ln(n(\frac{1}{2}, d))}{\big(\frac{1}{2}\big)^s+d^t}\ge\lim_{d\to\infty}\frac{\ln\frac{3}{4}+d\oz_d}{d^t}=\lim_{d\to\infty}d^{1-t}\oz_d\ge 0.$$
Hence \begin{equation}\label{3.21} \lim_{d\to
\infty}d^{1-t}\oz_d=0.
\end{equation}

Next we show \eqref{1.8} holds. We set
\begin{equation}\label{3.22} u_k:=\max(\oz_k, \frac{1}{2k}),
\qquad \text{and}\qquad
s_k:=\frac{1}{2}\big(\ln^+\frac{1}{u_k}\big)^{-1}, \qquad k\in
\Bbb N.
\end{equation}
 By \eqref{3.12-0} we have
\begin{equation*}
1-\va^2\le\sum_{k=1}^{n(\va,
d)}\lz_{d,k}\le\bigg(\sum_{k=1}^{n(\va, d)}
\lz_{d,k}^{1+s_d}\bigg)^{\frac{1}{1+s_d}}n(\va,
d)^{\frac{s_d}{1+s_d}}\le\Big(\sum_{k=1}^{\infty}\lz_{d,k}^{1+s_d}\Big)^{\frac{1}{1+s_d}}n(\va,
d)^{\frac{s_d}{1+s_d}}.
\end{equation*}
It follows that
$$n(\va, d)\ge(1-\va^2)^{\frac{1+s_d}{s_d}}\bigg(\sum_{k=1}^{\infty}\lz_{d,k}^{1+s_d}\bigg)^{\frac{-1}{s_d}}=(1-\va ^2)^{\frac{1+s_d}{s_d}}\prod_{k=1}^{d}\frac{(1-\oz_k^{1+s_d})^{\frac{1}{s_d}}}{(1-\oz_k)^{1+\frac{1}{s_d}}}.$$
We note that the function
$f(x)=\ln\big(\frac{1-x^{1+s_d}}{(1-x)^{1+s_d}}\big)$  is
monotonically increasing in $x\in (0, 1)$ due to the fact that
$f'(x)=\frac{(1+s_d)(1-x^{s_d})}{(1-x^{1+s_d})(1-x)}>0$ for
$x\in(0, 1)$. We have
\begin{equation}\label{3.23}
\begin{split}
\ln \big(n(\frac{1}{2}, d)\big)&\ge\frac{1+s_d}{s_d}\ln\frac{3}{4}+\frac{1}{s_d}\sum_{k=1}^d\ln \Big(\frac{1-\oz_k^{1+s_d}}{(1-\oz_k)^{1+s_d}}\Big)\\
&\ge\frac{1}{s_d}\ln\frac{3}{4}+\frac{d}{s_d}\ln \Big(\frac{1-\oz_d^{1+s_d}}{(1-\oz_d)^{1+s_d}}\Big)\\
&\ge\frac{1}{s_d}\ln\frac{3}{4}+\frac{d}{s_d}\ln \Big(1+\frac{\oz_d-\oz_d^{1+s_d}}{1-\oz_d}\Big)\\
&\ge\frac{1}{s_d}\ln\frac{3}{4}+
\frac{d\big(\oz_d-\oz_d^{1+s_d}\big)}{s_d\big(1-\oz_d\big)}\ln2,
\end{split}
\end{equation}where in the last inequality we used the inequality
$\ln(1+x)\ge x\ln 2,\ x\in[0,1]$. By \eqref{3.22} we get
\begin{equation}\label{3.24}
\frac{1}{s_d}=2\ln^+\big(\frac{1}{u_d}\big)\le 2\ln^+(2d)
\end{equation}
and
$$\lim_{d\to\infty}\frac{1}{d^ts_d}=\lim_{d\to\infty}\frac{2\ln^+(2d)}{d^t}=0.$$
Since $(s,t)$-WT holds for $s>0$ and $t\in (0, 1)$, we conclude by
\eqref{3.23} and \eqref{3.21} that
\begin{equation*}
\begin{split}
0=\lim_{d\to\infty}\frac{\ln \big(n(\frac{1}{2}, d)\big)}{d^t}&\ge\lim_{d\to\infty}\frac{\ln\frac{3}{4}}{d^ts_d}+\lim_{d\to\infty}\frac{d\big(\oz _d-\oz_d^{1+s_d}\big)}{d^ts_d\big(1-\oz_d\big)}\\
&\ge\lim_{d\to\infty}\frac{d^{1-t}}{s_d}\big(\oz _d-\oz_d^{1+s_d}\big)\ge0,
\end{split}
\end{equation*}
which yields that
\begin{equation}\label{3.25}
\lim_{d\to\infty}\frac{d^{1-t}}{s_d}\big(\oz _d-\oz_d^{1+s_d}\big)=0.
\end{equation}
Applying the mean value theorem to the function $g(x)=a^{1+x}(a\in
(0, 1))$, we get for some $\tz\in(0, 1)$, $$a-a^{1+x}=xa^{1+\tz
x}\ln \frac{1}{a}\le xa\ln\frac{1}{a}.$$ It follows  that
$$0\le\lim_{d\to\infty}\frac{d^{1-t}\big(\frac{1}{2d}-\big(\frac{1}{2d}\big)^{1+s_d}\big)}{s_d}\le
 \lim_{d\to\infty}d^{1-t}\big(\frac{1}{2d}\big)\ln(2d)=\lim_{d\to\infty}\frac{1}{2}d^{-t}\ln(2d)=0,$$
which gives
\begin{equation}\label{3.26}
\lim_{d\to\infty}\frac{d^{1-t}\big(\frac{1}{2d}-\big(\frac{1}{2d}\big)^{1+s_d}\big)}{s_d}=0.
\end{equation}
We remark that the function $u(x)=x-x^{1+s_d}$ is monotonically
increasing in $[0,\big(\frac{1}{1+s_d}\big)^{s_d}]\supset
(0,\frac{1}{e})$ and $\lim\limits_{d\to \infty}u_d=0$. By
\eqref{3.25} and \eqref{3.26}, we have
\begin{equation} \label{3.27}
\lim_{d\to\infty}\frac{d^{1-t}\big(u_d-u_d^{1+s_d}\big)}{s_d}=0.
\end{equation}
Using the mean value theorem, we conclude for some $\tz\in(0, 1)$ that,
$$u_d-u_d^{1+s_d}=u_du_d^{\tz s_d}s_d\big(\ln\big(\frac{1}{u_d}\big)\big)
\ge u_ds_d\big(\ln\big(\frac{1}{u_d}\big)\big)u_d^{s_d}=
u_ds_d\big(\ln\big(\frac{1}{u_d}\big)\big)e^{{\frac{-\ln\frac{1}{u_d}}{2\ln^+\frac{1}{u_d}}}}.$$
It follows from \eqref{3.27} that
\begin{equation*}
\begin{split}
0=\lim_{d\to\infty}\frac{d^{1-t}\big(u_d-u_d^{1+s_d}\big)}{s_d}&\ge\lim_{d\to\infty}d^{1-t}u_d\big(\ln\big(\frac{1}{u_d}\big)\big)
\lim _{d\to\infty}e^{{\frac{-\ln\frac{1}{u_d}}{2\ln^+\frac{1}{u_d}}}}\\
&=e^{-1/2}\lim_{d\to\infty}d^{1-t}u_d\big(\ln\big(\frac{1}{u_d}\big)\big)\ge
0,
\end{split}
\end{equation*}
which implies that
$$\lim_{d\to\infty}d^{1-t}u_d\big(\ln\big(\frac{1}{u_d}\big)\big)=0.$$
By the monotonically of the function $h(x)=x\ln \frac{1}{x}$ in $x\in(0, \frac{1}{e})$, we get
$$0\le \lim_{d\to \infty}d^{1-t}\oz_d\,\ln^+ \big(\frac{1}{\oz_d}\big)\le \lim_{d\to \infty}d^{1-t}u_d\big(\ln \big(\frac{1}{u_d}\big)\big)=0,$$
which combining with \eqref{2.4}, gives \eqref{1.8}.

On the other hand, we suppose that \eqref{1.8} holds. We want to
show that $(s, t)$-WT holds. We have for any $k\in \Bbb N$,
$$k\lz_{d, k}^{1-s_d}\le \sum_{j=1}^k \lz_{d,j}^{1-s_d}\le \sum_{j=1}^\infty \lz_{d,j}^{1-s_d},$$
so that \begin{equation}\label{3.28}
\lz_{d, k}\le \frac{\Big(\sum_{j=1}^\infty \lz_{d,j}^{1-s_d}\Big)^{\frac{1}{1-s_d}}}{k^{\frac{1}{1-s_d}}},
\end{equation}
where $s_d$ is given by \eqref{3.22}. Clearly,
\begin{equation}\label{3.29}
\sum_{k=n+1}^\infty\frac{1}{k^{\frac{1}{1-s_d}}}\le \int_n^\infty\frac{1}{x^{\frac{1}{1-s_d}}}dx=\frac{1-s_d}{s_d}n^{\frac{-s_d}{1-s_d}}.
\end{equation}
Combining \eqref{3.28} with \eqref{3.29} we conclude that
\begin{equation}\label{3.30}
\sum_{k=n+1}^\infty\lz_{d, k}\le
\sum_{k=n+1}^\infty\frac{1}{k^{\frac{1}{1-s_d}}}\Big(\sum_{j=1}^\infty
\lz_{d,j}^{1-s_d}\Big)^{\frac{1}{1-s_d}}\le
\Big(\frac{1-s_d}{s_d}\Big)n^{\frac{-s_d}{1-s_d}}\Big(\sum_{k=1}^\infty
\lz_{d,k}^{1-s_d}\Big)^{\frac{1}{1-s_d}}.
\end{equation}
Setting
$$ n=\left \lfloor\va^{\frac{-2(1-s_d)}{s_d}}\big(\frac{1-s_d}{s_d}\big)^{\frac{1-s_d}{s_d}}\Big(\sum_{k=1}^\infty \lz_{d,k}^{1-s_d}\Big)^{\frac{1}{s_d}}\right \rfloor+1$$
in \eqref{3.30}, we have
$$\sum_{k=n+1}^\infty\lz_{d, k}\le\va^2.$$

Therefore from the definition of $n(\va, d)$, and the inequality $\left \lfloor x\right \rfloor+1\le 2x$ for $x>1$, we get
\begin{equation*}
\begin{split}
n(\va, d)&\le\left \lfloor\va^{\frac{-2(1-s_d)}{s_d}}\big(\frac{1-s_d}{s_d}\big)^{\frac{1-s_d}{s_d}}\Big(\sum_{k=1}^\infty \lz_{d,k}^{1-s_d}\Big)^{\frac{1}{s_d}}\right \rfloor+1\\
&\le 2\va^{\frac{-2(1-s_d)}{s_d}}\big(\frac{1-s_d}{s_d}\big)^{\frac{1-s_d}{s_d}}\Big(\sum_{k=1}^\infty \lz_{d,k}^{1-s_d}\Big)^{\frac{1}{s_d}}.
\end{split}
\end{equation*}
It follows from \eqref{3.24} that
\begin{align}
\ln n(\va, d)&\le
\ln2+\frac{2(1-s_d)}{s_d}\ln(\va^{-1})+\frac{1-s_d}{s_d}
\ln\big(\frac{1-s_d}{s_d}\big)+\frac{1}{s_d}\ln
\Big(\sum_{k=1}^\infty \lz_{d,k}^{1-s_d}\Big)\notag \\
& \le \ln2+\frac{2}{s_d}\ln(\va^{-1})+\frac{1}{s_d}
\ln\big(\frac{1}{s_d}\big)+\frac{1}{s_d}\ln \Big(\sum_{k=1}^\infty
\lz_{d,k}^{1-s_d}\Big)\notag \\ &\label{3.31}
\le\ln2+4\ln^+(2d)\ln(\va^{-1})+2\ln^+(2d)\ln\big(2\ln^+(2d)\big)+\frac{1}{s_d}\ln
\Big(\sum_{k=1}^\infty \lz_{d,k}^{1-s_d}\Big)\\
&\le\ln2+\big(2\ln^+(2d)\big)^2+\big(\ln(\va^{-1})\big)^2\notag\\
&+2\ln^+(2d)\ln\big(2\ln^+(2d)\big)+\frac{1}{s_d}\ln
\Big(\sum_{k=1}^\infty \lz_{d,k}^{1-s_d}\Big).\notag
\end{align}
Note that
$$\lim_{\frac{1}{\va}+d\to\infty}\frac{\ln2+\big(2\ln^+(2d)\big)^2+\big(\ln(\va^{-1})\big)^2+2\ln^+(2d)\ln \big(2\ln^+(2d)\big)}{\big(\frac{1}{\va}\big)^s+d^t}=0.$$
In order to show that $(s,t)$-WT holds, it suffices to prove that
$$\lim_{d\to\infty}\frac{\ln \Big(\sum_{k=1}^\infty \lz_{d,k}^{1-s_d}\Big)}{d^ts_d}=0.$$
We recall that
$$\ln \Big(\sum_{k=1}^\infty \lz_{d,k}^{1-s_d}\Big)=\sum_{k=1}^d \ln\Big(\frac{(1-\oz_k)^{1-s_d}}{1-\oz_k^{1-s_d}}\Big).$$
Note that $v(x)=\ln\frac{(1-x)^\az}{1-x^\az} (\az\in(0, 1))$ is
increasing in $(0,1)$ due to the fact that
$v'(x)=\frac{\az(x^{\az-1}-1)}{(1-x^\az)(1-x)}>0$. We get
\begin{equation*}
\begin{split}
\frac{\ln \Big(\sum_{k=1}^\infty \lz_{d,k}^{1-s_d}\Big)}{d^ts_d}&\le \frac{\sum_{k=1}^d \ln\Big(\frac{(1-u_k)^{1-s_d}}{1-u_k^{1-s_d}}\Big)}{d^ts_d}\\
&=\frac{\sum_{k=1}^d\ln\big(\frac{1}{1-u_k}\big)}{d^t}+\frac{\sum_{k=1}^d \ln\Big(1+\frac{u_k^{1-s_d}-u_k}{1-u_k^{1-s_d}}\Big)}{d^ts_d}\\
&\le\frac{\sum_{k=1}^d\ln\big(\frac{1}{1-u_k}\big)}{d^t}+\frac{\sum_{k=1}^d \frac{u_k^{1-s_d}-u_k}{1-u_k^{1-s_d}}}{d^ts_d}\\
&=:I_{1,d}+I_{2,d},
\end{split}
\end{equation*}
where
\begin{equation}\label{3.25-0}I_{1,d}=\frac{\sum_{k=1}^d\ln\big(\frac{1}{1-u_k}\big)}{d^t}\quad
\text{and}\quad I_{2,d}=\frac{\sum_{k=1}^d
\frac{u_k^{1-s_d}-u_k}{1-u_k^{1-s_d}}}{d^ts_d}.\end{equation} By
\eqref{1.8} and \eqref{2.4}, we have
$$\lim_{d\to\infty}d^{1-t}\oz_d\ln^+\big(\frac{1}{\oz_d}\big)=0,$$which
combining the equality $$ \lim_{d\to\infty}d^{1-t}\big(\frac
{1}{2d}\big)\ln(2d)=0,$$ yields
$$\lim_{d\to\infty}d^{1-t}u_d\ln^+\big(\frac{1}{u_d}\big)=0.$$
We have by the Stolz theorem
$$\lim_{d\to\infty}I_{1, d}=\lim_{d\to\infty}\frac{\ln\big(\frac{1}{1-u_d}\big)}{d^t-(d-1)^t}=\lim_{d\to\infty}\frac{u_d}{d^{t-1}t}=0.$$
Applying the mean value theorem, we obtain for some $\tz\in(0, 1)$,
\begin{align}
u_k^{1-s_d}-u_k&=s_d u_k^{1-\tz s_d}\ln \big(\frac{1}{u_k}\big)\le s_d\,u_k\ln \big(\frac{1}{u_k}\big)\,u_k^{-s_d}\notag\\
&\le s_d\,u_k \ln \big(\frac{1}{u_k}\big)\,e^{\frac{\ln
(\frac{1}{u_k})}{2\ln ^+(\frac{1}{u_k})}}\le e^{1/2}\,s_d\,u_k\ln
(\frac{1}{u_k}).\label{3.32}
\end{align}
It follows from \eqref{3.25-0}, \eqref{3.32} and the inequality
$$1-u_k^{1-s_d}\ge 1-u_1^{1-s_d}\ge 1-u_1^{1-s_1}>0$$ that
$$I_{2,d}\le \frac{\sum_{k=1}^de^{1/2}u_k\big(\ln
\big(\frac{1}{u_k}\big)\big)s_d}{(1-u_1^{1-s_1})d^ts_d}
=C\frac{\sum_{k=1}^du_k\ln \big(\frac{1}{u_k}\big)}{d^t},$$ where
$C=\frac{e^{1/2}}{1-u_1^{1-s_1}}$. By the Stolz theorem we get
\begin{equation*}
\begin{split}
0\le\lim_{d\to \infty}I_{2,d}&\le C\lim_{d\to
\infty}\frac{\sum_{k=1}^du_k\ln \big(\frac{1}{u_k}\big)}{d^t}=
C\lim_{d\to \infty}\frac{u_d\ln \big(\frac{1}{u_d}\big)}{d^t-(d-1)^t}\\
&=\frac{C}{t}\lim_{d\to \infty}d^{1-t}u_d\ln
\big(\frac{1}{u_d}\big)=0.
\end{split}
\end{equation*}
We obtain further
\begin{equation}\label{3.25-1}\lim_{d\to\infty}\frac{\ln
\Big(\sum_{k=1}^\infty
\lz_{d,k}^{1-s_d}\Big)}{d^ts_d}=0.\end{equation} We conclude that
if \eqref{1.8} holds, then
$$\lim_{\frac1 \vz+d\to\infty}\frac{\ln(n(\va,
d))}{\big(\frac{1}{\va}\big)^s+d^t}=0,$$ which means that $(s,
t)$-WT holds. This finishes the proof of (iv).

\vskip 2mm

(v) The proof of (v) follows from (iv) and Lemma 3.2 immediately.

\vskip 2mm

(vi) Suppose that $\lim\limits_{j\to\infty}\ga_j^2>0$. Then
$\lim\limits_{j\to\infty}\oz_{\ga_j}=2A>0$. There there exists an
$N\in\Bbb N$ such that $$\ln(3/4)+d\oz_{\gz_d}\ge d\,A$$for any
$d>N$. By \eqref{3.12} we have for $\vz\in(0,1/2)$,
$$\ln n(\vz,d)\ge\ln n(1/2,d)\ge \ln (3/4)+d\,\oz_{\gz_d}\ge d\,A.$$
It follows that
$$n(\vz,d)\ge (e^A)^d,\ \ \vz\in(0,1/2],\ d>N.$$This means that $App$ suffers from the curse of
dimensionality. (vi) is proved. \vskip 2mm

The proof of  Theorem \ref{thm1} is completed. $\hfill\Box$

\

\noindent{\it \textbf{Proof of Theorem \ref{thm2}}.}

(1) We remark that we have the same results about EXP and UEXP  in
the worst and average case settings.  Indeed, using \eqref{2.8}
and the method in the proof of \cite[Theorem
 4.1]{LX}, we obtain that
${\rm App}=\{ {\rm App}_d\}$ is EXP  iff $I=\{I_d\}$ is EXP with
the same exponent. This completes the proof of (i).

\vskip 2mm

(2) Based on the results of \cite{SW} and \eqref{2.8}, we get that
EC-SPT,  EC-PT,  and EC-QPT do not hold for all shape parameters.
(ii) is proved.

\vskip 2mm

(3) According to \cite[Theorems 3.2 and 4.2]{X2} and \cite[Theorem
3.2]{LXD}, we know that we have the same results in the worst and
average case settings  using ABS concerning EC-WT, EC-UWT, and
EC-$(s,t)$-WT for $0<s\le 1$ and $t>0$. This implies that (v),
(vi) and (viii) hold. We always have EC-$(s,t)$-WT for $0<s\le 1$
and $t>1$. This yields EC-$(s,t)$-WT for $s> 1$ and $t>1$. Hence
(iii) holds. This completes the proofs of (iii), (v), (vi), and
(viii).

\vskip 2mm

(4) If EC-$(s,1)$-WT with $s\ge 1$ holds, then $(s,1)$-WT with
$s\ge 1$ holds.  By Theorem 1.1 (iii), we have
$\lim\limits_{j\to\infty}\gz_j^2=0.$

On the other hand, if $\lim\limits_{j\to\infty}\gz_j^2=0$, then
EC-WT holds and hence, EC-$(s,1)$-WT with $s\ge 1$ holds. This
completes the proof of (iv).

\vskip 2mm

(5) If EC-$(s,t)$-WT with $s> 1$ and $t<1$ holds, then $(s,t)$-WT
with $s> 1$ and $t<1$ holds. By Theorem 1.1 (iv), we have
\eqref{1.12}.

On the other hand, if \eqref{1.12} holds, then by \eqref{3.31} we
have
$$\ln n(\vz,d)\le \ln2+4\ln^+(2d)\ln(\va^{-1})+2\ln^+(2d)\ln\big(2\ln^+(2d)\big)+\frac{1}{s_d}\ln
\Big(\sum_{k=1}^\infty \lz_{d,k}^{1-s_d}\Big),$$ where $s_d$ is
given by \eqref{3.22}. By \eqref{3.25-1}, we obtain
\begin{equation}\label{3.35}\lim_{d\to\infty}\frac{\ln2+2\ln^+(2d)\ln\big(2\ln^+(2d)\big)+\frac{1}{s_d}\ln
\Big(\sum_{k=1}^\infty
\lz_{d,k}^{1-s_d}\Big)}{d^t}=0.\end{equation} For $s>1$, by the
Young inequality $ab\le \frac{a^p}p+\frac{b^{p'}}{p'},\ a,b\ge0, \
1/p+1/p'=1$ with $p=\frac{1+s}2,\ p'=\frac{s+1}{s-1}$ we have
$$\lim_{\vz^{-1}+d\to\infty}\frac
{\ln^+(2d)\ln(\va^{-1})}{(1+\ln\vz^{-1})^s+d^t}=
\lim_{\vz^{-1}+d\to\infty}\frac
{\frac{(\ln\va^{-1})^{\frac{s+1}2}}p+ \frac
{(\ln^+(2d))^{p'}}{p'}} {(1+\ln\vz^{-1})^s+d^t}=0,$$ which
combining \eqref{3.31} and \eqref{3.35}, leads to
$$\lim_{\vz^{-1}+d\to\infty}\frac
{\ln n(\vz,d)}{(1+\ln\vz^{-1})^s+d^t}=0. $$This finishes the proof
of (vii).

\vskip 2mm

The proof of  Theorem \ref{thm2} is completed. $\hfill\Box$

\

\section*{Acknowledgments}
 The  authors were Supported by the
National Natural Science Foundation of China (Project no.
11671271) and the Beijing Natural Science Foundation (1172004).

\end{document}